\documentclass[a4paper]{article}
\usepackage{etex}
\usepackage{pgfplots}

\usepackage{enumerate}
\usepackage{amsmath}

\usepackage{amsthm}

\usepackage{amssymb}
\usepackage{latexsym}
\usepackage{float}
\usepackage{amssymb}
\usepackage{latexsym}
\usepackage{tikz}
\usepackage{tikz-cd}
\usepackage{braids}
\usetikzlibrary{
  knots,
  hobby,
  decorations.pathreplacing,
  shapes.geometric,
  calc
}

\newtheorem{Claim}{Claim}

\theoremstyle{definition}

\def\FF{{\mathbb F}}

\begin{document}

\title{A note on a paper by R. C. Alperin}

\author{
Jean Chessa and Luisa Paoluzzi
}
\date{\today}
\maketitle

\begin{abstract}

\vskip 2mm
We point out a mistake in the statement of Corollary 2 in R.~C.~Alperin's paper
on Selberg's lemma. 

\vskip 2mm


\vskip 2mm


\end{abstract}

\section{Introduction}

The aim of this short note is to point out an incorrect statement in Alperin's 
paper ``An elementary account of Selberg's lemma" which appeared in
L'Enseignement Math\'ematique \cite{Alp}. The main focus of the paper is to 
provide an accessible proof of Selberg's lemma, which affirms that any linear 
group defined over a field of characteristic zero contains a torsion free 
subgroup of finite index.

Selberg's lemma has several applications, both in algebra and geometry. Some 
straightforward consequences can also be found in Alperin's paper. 
Unfortunately, among these, part of the statement of Corollary 2 on page 272 is 
wrong. The corollary reads: 

\emph{The torsion subgroups of a finitely generated linear group $G$ are 
finite; moreover, in characteristic zero, these finite groups have bounded 
order.}

Although the corollary is correct in characteristic zero, the assertion is
false in positive characteristic. We believe that, in the proof, 
it is implicitely assumed that torsion subgroups of finitely generated linear 
groups are finitely generated. The rest of this note will be devoted to show
that this is however not always the case.


\section{Counterexamples to the statement}

The examples in the following claim are not new and can be found in
\cite[Example 26.9, page 779]{Nic}.

\begin{Claim}
Let $n\ge 3$ and $p$ a prime number. Let $\FF_p[t]$ be the ring of polynomials
in one variable over the field $\FF_p$ with $p$ elements. The special linear 
group $SL(n,\FF_p[t])$ is finitely generated and contains the infinite torsion 
group $H=\{I_n+xE_{1\,n}\mid x\in \FF_p[t]\}$, where $E_{1\,n}$ denotes the 
$n\times n$ matrix with the $(1,n)$ entry equal to $1$ and all other entries 
equal to $0$.
\end{Claim}

\begin{proof}
The group $SL(n,\FF_p[t])$ is clearly linear over the field $\FF_p(t)$ of 
positive characteristic $p$. For $n\ge 3$, it was shown to be finitely 
generated by Behr in \cite{Beh}.

Since $(I_n+xE_{1\,n})(I_n+yE_{1\,n}) =I_n+(x+y)E_{1\,n}$, $H$ is a subgroup of 
$SL(n,\FF_p[t])$ isomorphic to the additive Abelian $p$-group $\FF_p[t]$. This
is an infinite dimensional vector space over the field $\FF_p$, so it is a
non finitely generated group of exponent $p$. This shows in particular that it 
is an infinite torsion group.
\end{proof}

In positive characteristic there are thus torsion subgroups of finitely 
generated linear groups that are not finite, contradicting Corollary 2 of
Alperin's paper. On the other hand, by Burnside's theorem 
(see \cite[Corollary 3]{Alp}) finitely generated torsion subgroups are 
indeed finite, regardless of the characteristic.


\textsc{D\'partement de Math\'ematiques, ENS de Lyon, France}

{jean.chessa@ens-lyon.fr}

\textsc{Aix-Marseille Univ, CNRS, Centrale Marseille, I2M, UMR 7373,
13453 Marseille, France}

{luisa.paoluzzi@univ-amu.fr}


\footnotesize

\begin{thebibliography}{XXXXX}

\bibitem[Alp]{Alp} R. C. Alperin, 
\emph{An elementary account of Selberg’s lemma}, 
Enseign. Math. \textbf{33}, (1987), 269–273.

\bibitem[Beh]{Beh} H. Behr,
\emph{Endliche Erzeugbarkeit arithmetischer Gruppen \"uber 
Funktionenk\"orpern}, Inv. Math. \textbf{7}, (1969), 1-32.

\bibitem[Nic]{Nic} B. Nica,
\emph{Three theorems on linear groups},
Appendix to ``Geometric Group Theory", by C. Dru\c{t}u and M. Kapovic, 
Colloquium Publications \textbf{63}, Amer. Math. Soc. Providence, RI (2018),
777-786.


\end{thebibliography}
\end{document}